\documentclass[a4paper,12pt]{article}

\usepackage{amsfonts}
\usepackage{amscd,color}
\usepackage{amsmath,amsfonts,amssymb,amscd}
\usepackage{indentfirst,graphicx,epsfig}
\usepackage{graphicx}
\usepackage{tikz}

\input{epsf}
\usepackage{epstopdf}
\usepackage{caption}

\newtheorem{theorem}{Theorem}[section]

\newtheorem{lemma}[theorem]{Lemma}

\newtheorem{corollary}[theorem]{Corollary}
\newtheorem{proposition}[theorem]{Proposition}

\tikzstyle{snode}=[circle,draw=black,fill=white,thick, inner sep=0pt ,minimum size=1.2mm]
\tikzstyle{bnode}=[circle ,draw=black,fill=black,thick, inner sep=0pt ,minimum size=1.2mm]

\newenvironment {proof} {\noindent{\em Proof.}}{\hspace*{\fill}$\Box$\par\vspace{4mm}}
\newcommand{\ml}{l\kern-0.55mm\char39\kern-0.3mm}

\newenvironment{theorem-non}[1]{\trivlist \item [\hskip \labelsep {\bf #1}]\ignorespaces\it}{\endtrivlist}

\begin{document}
\title{\textbf{Proper rainbow connection number of graphs}}

\author{\small \ Trung Duy Doan$^{1}$ ,\ Ingo Schiermeyer$^{2}$
\\[0.2cm]
\small $^{1}$ School of Applied Mathematics and Informatics \\
\small Hanoi University of Science and Technology, Hanoi, Vietnam \\
\small $^{2}$Institut f\"ur Diskrete Mathematik und Algebra \\
\small Technische Universit\"at Bergakademie Freiberg \\
\small 09596 Freiberg, Germany\\[0.2cm]
{\small Email: trungdoanduy@gmail.com, Ingo.Schiermeyer@tu-freiberg.de}\\
}

\date{\today}

\maketitle

\begin{abstract}
A path in an edge-coloured graph is called \emph{rainbow path} if its edges receive pairwise  distinct colours. An edge-coloured graph is said to be \emph{rainbow connected} if any two distinct vertices of the graph are connected by a rainbow path. The
minimum $k$ for which there exists such an edge-colouring is the rainbow connection number $rc(G)$ of $G.$
Recently, Bau et al. \cite{BJJKM2018} introduced this concept with the additional requirement that the edge-colouring must be proper. 
The \emph{proper rainbow connection number} of $G$, denoted by $prc(G)$, is the minimum number of colours needed in order to make it properly rainbow connected.

In this paper we first prove an improved upper bound $prc(G) \leq n$ for every connected graph $G$ of order $n \geq 3.$ Next we show that the difference $prc(G) - rc(G)$ can be arbitrarily large. Finally, we present several sufficient conditions for graph classes satisfying $prc(G) = \chi'(G).$

\medskip

\end{abstract}

\section{Introduction}
 We use \cite{West2001} for terminology and notation not defined here and consider only simple, finite and undirected graphs. Let $G$ be a graph. We denote by $V(G), E(G), n, m, \Delta(G),$ $diam(G)$ the vertex set, the edge set, number of vertices, number of edges, maximum degree, and diameter of $G$, respectively. Let $K_n,C_n,P_n$ be a complete graph, a cycle and a path on $n$ vertices, respectively. By $N_G(u)$ we denote the set of neighbours of a vertex $u\in V(G)$ and by $d(u)$ its degree $G$. Let us denote by $d(u,v)$ and $d(uPv)$ the distance between two vertices $u,v$  and the length of a $u,v$-path $P$, respectively. For each integer $n\geq4$, the wheel is defined as $W_n=C_n+K_1$, the join of $C_n$ and $K_1.$
 For simplifying notation, let $[k]$ be the set $\lbrace 1,2,\ldots,k\rbrace$ for some positive integer $k$.

Let $c: E(G) \rightarrow [k]$ be an edge-colouring of $G.$ If adjacent edges of $G$ receive different colours by $c$, then $c$ is a \emph{proper colouring}. The smallest number of colours needed in a proper colouring of $G$, denoted by $\chi'(G)$, is called the \emph{chromatic index} of $G$. Vizing et al. \cite{Vizing1964} proved that for any graph $G$, $\chi'(G)$ is either its maximum degree $\Delta(G)$ or $\Delta(G)+1$. If $\chi'(G)=\Delta(G)$, then $G$ is in \emph{class 1}. Otherwise, $G$ is in \emph{class 2}.

A path $P$ in an edge-coloured graph $G$ is called a \emph{rainbow path} if its edges have different colours. An edge-coloured graph $G$ is \emph{rainbow connected} if every two vertices are connected by at least one rainbow path in $G$. For a connected graph $G$, the \emph{rainbow connection number} of $G$, denoted by $rc(G)$, is defined as the smallest number of colours required to make it rainbow connected. The concept of rainbow connection was first introduced by Chartrand et al. \cite{Chartrand2008} and well-studied since then. Readers who are interested in this topic are referred to \cite{LiSh2013,Li2012}.

As an extension of proper colouring and motivated by rainbow connections of graphs, Bau et al. \cite{BJJKM2018} introduced the concept of proper rainbow connections in connected graphs. Let $G$ be a nontrivial connected graph. The proper edge-coloured graph $G$ is said to be \emph{properly rainbow connected} if any two vertices $u,v\in V(G)$ are connected by a rainbow path. The \emph{proper rainbow connection number} $prc(G)$ of a connected graph $G$ is the smallest number of colours needed to colour $G$ properly rainbow connected.

By the definition above, if an edge-coloured graph $G$ is properly rainbow connected, then $G$ is properly coloured and rainbow connected. Hence, the lower bound of the proper rainbow connection number was obtained by the following proposition.

\begin{proposition}[Bau et al. \cite{BJJKM2018}]
\label{Prop_Lower_Bau}
Let $G$ be a connected graph. Then $$diam(G)\leq rc(G)\leq prc(G)$$
and
$$\chi'(G)\leq prc(G).$$
\end{proposition}

On the other hand, if every edge of $G$ receives a distinct colour from $[m]$, where $m$ is the number edges of $G$, then $G$ is properly rainbow connected. By using Proposition \ref{Prop_Lower_Bau} and Vizing's Theorem in \cite{Vizing1964}, the proper rainbow connection number of an arbitrary connected graph is bounded as follows.
\begin{corollary}
\label{Cor_Bounded}
Let $G$ be a connected graph of size $m$. Then $$\max\lbrace rc(G),\chi'(G)\rbrace\leq prc(G)\leq m.$$
\end{corollary}

The authors in \cite{JLLM2019} determined some graphs with large proper rainbow connection number. First of all, they characterized all graphs whose proper connection numbers equal their size.

\begin{theorem}[Jiang et al. \cite{JLLM2019}]
\label{Thm_Jiang_prc_size}
Let $G$ be a connected graph of size $m$. Then $prc(G)=m$ if and only if $G$ is a tree or $K_3$.
\end{theorem}

After that, they also classified connected graphs whose proper connection numbers are close to the maximum possible value. Let $\mathcal{H'}$ and $\mathcal{H''}$ be two graph classes as shown in Figure \ref{Fig_Jiang_1}, where the order of $H'\in\mathcal{H'}$ is at least $4$ and the order of $H''\in\mathcal{H''}$ is at least $5$, respectively.

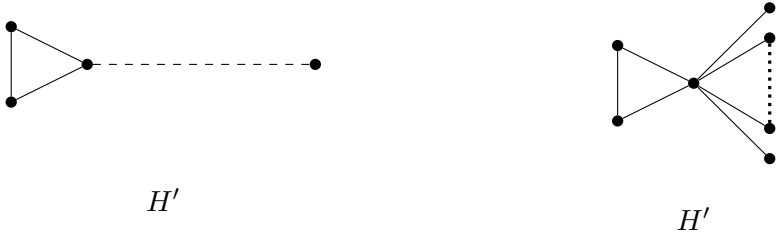
\begin{figure}
\begin{minipage}{0.45\textwidth}
\centering
\begin{tikzpicture}
\node [bnode] (v1) at (0,0) {};
\node [bnode] (v2) at (-1,0.5) {};
\node [bnode] (v3) at (-1,-0.5) {};
\node [bnode] (v4) at (3,0) {};
\draw (v1)--(v2)--(v3)--(v1);
\draw [dashed] (v1)--(v4);
\node [below] at (1,-1.5) {\small $H'$};
\end{tikzpicture}
\end{minipage}
\begin{minipage}{0.45\textwidth}
\centering
\begin{tikzpicture}
\node [bnode] (v1) at (0,0) {};
\node [bnode] (v2) at (-1,0.5) {};
\node [bnode] (v3) at (-1,-0.5) {};
\node [bnode] (v6) at (1,1) {};
\node [bnode] (v7) at (1,-1) {};
\node [bnode] (v4) at (1,0.6) {};
\node [bnode] (v5) at (1,-0.6) {};
\node [below] at (0,-1.5) {\small $H'$};
\draw (v1)--(v2)--(v3)--(v1);
\draw (v1)--(v4);
\draw (v1)--(v5);
\draw (v1)--(v7);
\draw (v1)--(v6);
\draw [very thick, dotted] (v4)--(v5);
\end{tikzpicture}
\end{minipage}
\caption{The graphs $H'\in\mathcal{H'}$ and $H''\in\mathcal{H''}$.}
\label{Fig_Jiang_1}
\end{figure}

\begin{theorem}[Jiang et al. \cite{JLLM2019}]
\label{Thm_Jiang_prc_close_size}
If $G$ is a connected graph of size $m$, then $prc(G)=m-1$ if and only if $G\in\mathcal{H'}$ or $G\in\mathcal{H''}$.
\end{theorem}

Next, the proper rainbow connection numbers of special graphs were considered by Bau et al. \cite{BJJKM2018} and Jiang et al. \cite{JLLM2019}.

\begin{theorem}[Bau et al. \cite{BJJKM2018}]
\label{Thm_Bau_Compl}
For each integer $n \geq 2$,
$$prc(K_n)=\chi'(K_n)=
\begin{cases}
n-1, &\mbox{if } n \mbox{ is even} \\
n, &\mbox{if } n \mbox{ is odd}
\end{cases}
$$
\end{theorem}

\begin{corollary}[Jiang et al. \cite{JLLM2019}]
\label{Cor_Jiang_Cyc}
For each interger $n\geq4$, $prc(C_n)=\lceil\frac{n}{2}\rceil.$
\end{corollary}

Chartrand et al. \cite{Chartrand2008} determined the rainbow connection numbers of complete graphs.

\begin{proposition}[Chartrand et al. \cite{Chartrand2008}]
\label{Pro_Chartrand_Compl}
Let $n\geq2$ be an integer. Then $rc(K_n)=1.$
\end{proposition}

By Theorem \ref{Thm_Bau_Compl} and Proposition \ref{Pro_Chartrand_Compl}, it can be readily seen that the difference $prc(G)-rc(G)$ can be arbitrarily large.

The \emph{cartesian product} of two graphs $G$ and $H$ written $G\square H$ is the graph with vertex set $V(G)\times V(H)$
specified by putting $(u_1,u_2)$ adjacent to $(v_1,v_2)$ if and only if (1) $u_1=u_2$ and $v_1v_2\in E(H)$, or
(2) $v_1=v_2$ and $u_1u_2\in E(G).$
Bau et al. \cite{BJJKM2018} and Jiang et al. \cite{JLLM2019} determined proper connection numbers of cartesian products by the following results.

\begin{proposition}[Bau et al. \cite{BJJKM2018}]
\label{Pro_Bau_Carte}
Let $n,p_1,\ldots,p_n>1$ be integers and \\$G=K_{p_1}\square\ldots\square K_{p_n}.$ Then
$$\sum_{i=1}^{n}(p_i-1)\leq prc(G)\leq\sum_{i=1}^{n}\chi'(K_{p_i})$$
\end{proposition}

\begin{theorem}[Jiang et al. \cite{JLLM2019}]
\label{Thm_Jiang_Carte}
Suppose that $n\geq1$, and $p_1,\ldots,p_n>1$ are integers. If $G=K_{p_1}\square\ldots\square K_{p_n}$, then $prc(G)=\chi'(G)$.
\end{theorem}


\section{Upper bounds}
In this section we will show improved upper bounds for the proper rainbow connection number of graphs.

The concept of rainbow connection was first introduced by Chartrand et al. \cite{Chartrand2008}. Moreover, they gave the relation between rainbow connection number of a connected graph and rainbow connection number of its spanning tree as follows.

\begin{proposition}[Chartrand et al. \cite{Chartrand2008}]
\label{Prop_Chartrand_rc}
Let $G$ be a connected graph and $T$ be a spanning tree of $G$. Then $rc(G)\leq rc(T)$.
\end{proposition}

By Proposition \ref{Prop_Chartrand_rc}, it can be readily seen that if $G$ has $n$ vertices, then $rc(G)\leq n-1$. Moreover, by Theorem \ref{Thm_Bau_Compl} and Corollary \ref{Cor_Jiang_Cyc}, proper connection numbers of complete graphs and cycles do not exceed their number of vertices. These facts are our motivation to improve the upper bound for the proper rainbow connection number as follows.

\begin{figure}
\begin{center}
\begin{tikzpicture}
\node [bnode] (w) at (0,0) {};
\node [above] at (0,0) {$w$};
\node [bnode] (w1) at (-2.5,-3) {};
\node [below] at (-2.5,-3) {$w_1$};
\node [bnode] (w2) at (-1,-3) {};
\node [below] at (-1,-3) {$w_2$};
\node [bnode] at (-0.5,-3) {};
\node [bnode] at (0,-3) {};
\node [bnode] at (0.5,-3) {};
\node [bnode] at (1,-3) {};
\node [bnode] (w_max) at (2.5,-3) {};
\node [below] at (2.5,-3) {$w_{\bigtriangleup(G)}$};
\draw (w1)--(w)--(w2);
\draw (w_max)--(w);
\node [left] at (-1.2,-1.5) {$1$};
\node [left] at (-0.4,-1.5) {$2$};
\node [right] at (1.2,-1.5) {$\bigtriangleup(G)$};
\end{tikzpicture}
\caption{Tree $K_{1,\bigtriangleup(G)}$}
\label{Figure_upper_bound}
\end{center}
\end{figure}

\begin{theorem}
\label{Thm_upper_bound}
Let $G$ be a nontrivial, connected graph of order $n$ and maximum degree $\Delta(G)$. Then $$\max\lbrace \Delta(G),diam(G)\rbrace\leq prc(G)\leq \chi'(G)+(n-1-\bigtriangleup(G))=
\begin{cases}
n,\mbox{ if }G\mbox{ is class 2}\\
n-1,\mbox{ if }G\mbox{ is class 1}
\end{cases}
$$
\end{theorem}

\begin{proof}
\label{proof_thm_upper_bound}
Since $\Delta(G)\leq\chi'(G)$ by Vizing's Theorem in \cite{Vizing1964} and $diam(G)\leq rc(G)$ by Chartrand et al. \cite{Chartrand2008}, the lower bound is easily obtained.

Next, we consider the upper bound. Since $G$ has maximum degree $\bigtriangleup(G)$, there exists a vertex $w\in V(G)$ such that $d_G(w)=\bigtriangleup(G)$. Let $N_G(w)=\lbrace w_1,w_2,...,w_{\bigtriangleup(G)}\rbrace$ be the neighbour set of $w\in V(G)$. We construct a tree $T \cong K_{1,\bigtriangleup(G)}$, which consists of the vertex $w$ as a root of the tree $K_{1,\bigtriangleup(G)}$ and all the vertices in the set $N_G(w)$. Let $c$ be a proper edge-colouring of $G$ with $\chi'(G)$ colours $\{1, 2, \ldots, \chi'(G)\}$. We may assume that the edges of $T$ have colours $1, 2, \ldots, \Delta(G)$. Since $G$ is connected, we can extend the tree $T$ to a spanning tree $T'$ of $G$ by properly adding $n-1-\Delta(G)$ edges. Now we recolour these $n-1-\Delta(G)$ edges by using $n-1-\Delta(G)$ new colours. This leads to an proper edge-colouring $c'$ of $G,$ since every new colour is used exactly once. Moreover, the tree $T'$ is rainbow coloured, which shows that $G$ is properly rainbow-connected.

By using Vizing's Theorem in \cite{Vizing1964}, $prc(G)\leq n$, if $G$ is class $2$ or $prc(G)\leq n-1$, if $G$ is class $1$. We obtain the result.
\end{proof}

Now, we improve the upper bound by requiring a structural condition.

\begin{proposition}
\label{Propo_improve_upper_bound}
Let $G$ be a connected graph with maximum degree $\Delta(G)\leq n-2$ and $w\in V(G)$ such that $d(w)=\Delta(G)$. If there is a Hamiltonian cycle in $G-N[w]$, then $prc(G)\leq\frac{n+\Delta(G)}{2}+1$.
\end{proposition}
\begin{proof}
Suppose that $G$ has a vertex $w$ with $d(w)=\Delta(G)$ and there is a Hamiltonian cycle, say $C$, in $G-N[w]$. Let us colour all the edges of $E(G)\setminus E(C)$ by $c:1\rightarrow[\chi'(G)]$ in order to make it a proper colouring. Next, we continue to colour all remaining edges of $C$ with $\lceil\frac{\vert V(C)\vert}{2}\rceil$ new colours. It can be readily seen that $G$ is properly rainbow connected using $\chi'(G)+\bigl\lceil\frac{\vert V(C)\vert}{2}\bigr\rceil$ colours.
Since $\Delta(G)\leq n-2$, it can be readily seen that $$prc(G)\leq\Delta(G)+1+\bigl\lceil\frac{n-1-\Delta(G)}{2}\bigr\rceil\leq\frac{n+\Delta(G)}{2}+1.$$

The result is obtained.
\end{proof}

\section{Estimating the difference $prc(G) - rc(G)$}

Next observe that

$$0 \leq prc(G) - rc(G) \leq n-1$$

for all connected graphs $G,$ where the upper bound is attained for $K_n$ if $n$ is odd. We now extend this observation as follows.

\begin{proposition}
Let $G$ be a connected graph of order $n \geq 2$ with clique number $\omega(G) \geq \frac{n+1}{2}.$ Then
$$prc(G) - rc(G) \geq 2\omega(G)-n-1.$$
\end{proposition}

\begin{proof}
If $\omega(G)=n,$ then the inequality is true. Hence we may assume that $\frac{n+1}{2} \leq \omega(G) \leq n-1.$
First observe that $rc(G) \leq n+1-\omega(G).$ To see this, take a clique of size $\omega(G)$ and colour all edges between its vertices by one colour. Next we add
$n-\omega(G)$ edges to obtain a spanning subgraph of $G.$ We colour each edge by a new colour and can colour all remaining edges arbitrarily. Then this colouring makes
$G$ rainbow connected. Since $G$ is connected and $\omega(G) \leq n-1,$ we deduce that $\Delta(G) \geq \omega(G).$ Hence by Theorem 2.2 we obtain $prc(G) \geq \omega(G).$
Now the inequality follows.
\end{proof}

Next we analyse the values of $prc(G)$ and $rc(G)$ for graphs with respect to their minimum degree.

\begin{theorem}
Let $G$ be a connected graph of order $n \geq 3$ and minimum degree $\delta = \delta(G). If$
\begin{enumerate}
\item $\delta \geq 3,$ then $rc(G) \leq \frac{3n}{4}$ \cite{S2009},
\item $\delta \geq 4,$ then $rc(G) \leq \frac{3n}{\delta + 1}+3$ \cite{CDRV2012}.
\end{enumerate}
\end{theorem}

Now observe that $\delta \geq \frac{3n}{\delta +1}+3$ if $\delta \geq 1 + \sqrt{3n+4}.$ With $prc(G) \geq \chi'(G) \geq \Delta(G) \geq \delta(G)$ we thus obtain

\begin{theorem}
Let $G$ be a connected graph of order $n$ and with $\delta(G) \geq 1 + \sqrt{3n+4}.$ Then

$$prc(G) - rc(G) \geq \delta - (\frac{3n}{\delta +1}+3)$$

\end{theorem}

Further observe that

$$0 \leq prc(G) - \chi'(G) \leq (n-1)-2= n-3$$

for all connected graphs with $n \geq 3,$ where the upper bound is attained for the path $P_n.$

Since $prc(G) \geq max\{rc(G),\chi'(G)\}$ by Corollary 1.2, it is natural to ask whether the difference
$prc(G) - max\{rc(G),\chi'(G)\}$ is unbounded as well. In our next Theorem we show that the difference $prc(G) - rc(G)$ can be arbitrarily large.

\begin{theorem}
\label{Thm_exist_graph_Delta_Diam}
Let $k,t$ be two integers, where $k\geq t\geq1$. There always exists a connected graph $G_{k,t}$ with $\Delta(G)=2t^2+1$ and $diam(G)=2t^2+1+k$ such that $prc(G)\geq rc(G)+t$.
\end{theorem}

\begin{figure}
\begin{center}
\begin{tikzpicture}
\node [bnode] (u1) at (0,0) {};
\node [above] at (0,0) {$v_1$};
\node [bnode] (u2) at (0,-3) {};
\node [below] at (0,-3) {$v_2$};
\node [bnode] (u3) at (1.5,-1.5) {};
\node [above] at (1.5,-1.5) {$u_3$};
\node [bnode] (u4) at (3.5,-1.5) {};
\node [above] at (3.5,-1.5) {$u_4$};
\node [bnode] (u5) at (5.5,-1.5) {};
\node [above] at (5.5,-1.5) {$u_5$};
\node [bnode] (u6) at (7.5,-1.5) {};
\node [above] at (7.5,-1.5) {$u_6$};

\draw (u1)--(u2)--(u3)--(u4)--(u5)--(u6);
\draw (u1)--(u3);
\node [left] at (0,-1.5) {$2$};
\node [above] at (0.75,-0.75) {$1$};
\node [below] at (0.75,-2.25) {$3$};
\node [above] at (2.5,-1.5) {$2$};
\node [above] at (4.5,-1.5) {$4$};
\node [above] at (6.5,-1.5) {$5$};
\end{tikzpicture}
\caption{Graph $G_{1,1}$ with $rc(G)=4$ and $prc(G)=rc(G)+1=5$}
\label{Figure_3}
\end{center}
\end{figure}
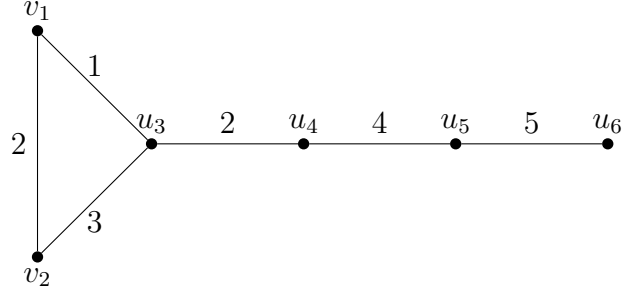

\begin{proof}
Firstly, if $k=t=1$, then we take a connected graph with $\Delta(G)=3$ and $diam(G)=4$ as shown in Figure \ref{Figure_3}. Clearly, $rc(G)=4$ and $prc(G)=5=rc(G)+t$.

Now, we consider $t\geq2$. Let $W_{2t^2}$ be a wheel consisting of a cycle $C=v_1\ldots v_{2t^2}v_1$ and a center vertex $v$. Let $G$ be a connected graph constructed from $W_{2t^2}$ and a path $P=u_{2t^2+1}\ldots u_{4t^2+1+k}$ of order $2t^2+1+k$ by identifying $v$ and $u_{2t^2+1}$ as shown in Figure \ref{Figure_4}. It can be readily seen that $\Delta(G)=2t^2+1$ and $diam(G)=2t^2+1+k$. Hence, $rc(G)\geq 2t^2+1+k$. Let us define a colouring $c$ with $2t^2+1+k$ colours to colour all the edges of $G$ as follows.
$$
c(e)=
\begin{cases}
1 & \mbox{if} \hspace{0.5em} e=vv_i, \forall i\in[2t^2]\\
i+1-2t^2 & \mbox{if} \hspace{0.5em} e=u_iu_{i+1}, \forall i\in[2t^2+1,4t^2+k]\\
i & \mbox{if} \hspace{0.5em} e=v_iv_{i+1}, \forall i\in[t^2]\\
i-t^2 & \mbox{if} \hspace{0.5em} e=v_iv_{i+1},\forall i\in[t^2+1,2t^2-1]\\
t^2 & \mbox{if}\hspace{0.5em} e=v_{2t^2}v_1
\end{cases}
$$
It can be readily seen that $G$ is rainbow connected with $2t^2+1+k$ colours. Thus, $rc(G)\leq 2t^2+1+k$. So we deduce that $rc(G)=2t^2+1+k$.

Next, we show that $prc(G)\geq 2t^2+1+k+t$.

Suppose that $prc(G)\leq 2t^2+k+t$. Then there is a colouring $c$ with $2t^2+k+t$ colours which makes $G$ properly rainbow connected. Since $P$ is the only path from $u_{2t^2+1}$
to $u_{4t^2+1+k},$ $P$ is a rainbow path. Hence we may assume that its $2t^2+k$ edges are coloured with the colours $1, 2, \ldots, 2t^2+k.$ Since $G$ is properly rainbow connected, all the edges that are incident to $v$ receive distinct colours. Moreover, every rainbow path from $u_{4t^2+1+k}$ to a vertex $v_i, 1 \leq i \leq 2t^2,$ uses exactly one edge $vv_j,$ whose colour is distinct from
$1, 2, \ldots, 2t^2+k.$ Hence we may assume that $p$ of these edges, where $1 \leq p \leq t,$ have a colour from the set $2t^2+k+1, \ldots, 2t^2+k+t.$
Suppose first that $p=1.$ We may assume that $vv_1$ has the only colour from the set $2t^2+k+1, \ldots, 2t^2+k+t.$ Then any rainbow path from $v_{t^2+1}$ to $u_{4t^2+1+k}$ has
$t^2 + 1 + (2t^2+k) > 2t^2+k+t$ edges, a contradiction.
Next suppose that $p \geq 2.$ So there are $p$ integers
$1 \leq i_1 < i_2 < \ldots < i_p \leq 2t^2$ such that the edges $vv_{i_1}, \ldots, vv_{i_p}$ have these $p$ colours. So there is a partition of the cycle $C$ into $p$ paths, each connecting
$v_{i_j}$ with $v_{i_{j+1}}$ along the cycle of length $|i_{+1}-i_j|$ (modulo $2t^2$). Hence the longest of these $p$ paths has length at least $\frac{2t^2}{p} \geq \frac{2t^2}{t} = 2t.$
We may assume that the path from $v_{i_1}$ to $v_{i_2}$ has length at least $2t$ and that $i_1=1.$ Then any rainbow path from $v_{t+1}$ to $u_{4t^2+1+k}$ has at least $t+1+2t^2+k$ edges,
a contradiction.

\end{proof}

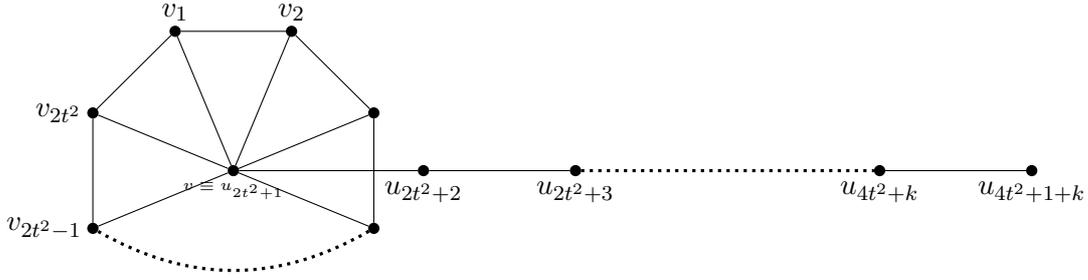
\begin{figure}
\begin{center}
\begin{tikzpicture}
\def\rot{8}
\node [bnode] (v) at (0,0) {};
\node [below] at (0,0) {\tiny $v\equiv u_{2t^2+1}$};
\node [bnode] (v1) at ({45*4.5:2}) {};
\node [left] at ({45*4.5:2}) {\small $v_{2t^2-1}$};
\node [bnode] (v2) at ({45*3.5:2}) {};
\node [left] at ({45*3.5:2}) {\small $v_{2t^2}$};
\node [bnode] (v3) at ({45*2.5:2}) {};
\node [above] at ({45*2.5:2}) {\small $v_1$};
\node [bnode] (v4) at ({45*1.5:2}) {};
\node [above] at ({45*1.5:2}) {\small $v_2$};
\node [bnode] (v5) at ({45*0.5:2}) {};
\node [bnode] (v6) at ({45*-0.5:2}) {};
\draw (v1)--(v2)--(v3)--(v4)--(v5)--(v6);
\draw (v)--(v1);
\draw (v)--(v2);
\draw (v)--(v3);
\draw (v)--(v4);
\draw (v)--(v5);
\draw (v)--(v6);
\draw (v6) [dotted, very thick, bend left] to (v1);
\node [bnode] (u1) at (2.5,0) {};
\node [below] at (2.5,0) {\small $u_{2t^2+2}$};
\node [bnode] (u2) at (4.5,0) {};
\node [below] at (4.5,0) {\small $u_{2t^2+3}$};
\node [bnode] (u3) at (8.5,0) {};
\node [below] at (8.5,0) {\small $u_{4t^2+k}$};
\node [bnode] (u4) at (10.5,0) {};
\node [below] at (10.5,0) {\small $u_{4t^2+1+k}$};
\draw (v)--(u1)--(u2);
\draw [very thick, dotted] (u2)--(u3);
\draw (u3)--(u4);
\end{tikzpicture}
\caption{The graph $G_{k,t}$ for $t \geq 2$}
\label{Figure_4}
\end{center}
\end{figure}

\section{Graph classes with $prc(G) = rc(G)$}

\begin{proposition}
Let $C_n$ by a cycle of order $n \geq 4.$ Then
$$prc(G) = rc(G) = \lceil\frac{n}{2}\rceil.$$
\end{proposition}

Next observe that any colouring, which makes a tree rainbow connected, is a proper colouring. So we deduce that

\begin{proposition}
Let $T$ by a tree of order $n \geq 2.$ Then
$$prc(G) = rc(G) = n-1.$$
\end{proposition}

Starting with a given tree $T$ we can generate a large variety of classes of graphs satisfying $prc(G)=rc(G).$ For example we can attach to each leaf of
$T$ a finite number of cycles each of length at least 4.

Another class of graphs is described in \cite{JLLM2019}. Here $g(G)$ denotes the girth of $G.$

\begin{proposition}
Let $G$ be a connected graph with $rc(G) < g(G)-2.$ Then
$$prc(G) = rc(G).$$
\end{proposition}


\section{Graph classes with $prc(G) = \chi'(G)$}

The proper rainbow connection numbers of complete graphs $K_n$ are determined in Theorem \ref{Thm_Bau_Compl}. Now we consider the proper rainbow connection number of connected graphs whose diameter is $2$.

\begin{proposition}
\label{Propo_diam}
Let $G$ be a connected graph of order $n \geq 3.$ If $diam(G)=2$, then $prc(G)=\chi'(G)$.
\end{proposition}

\begin{proof}
Let $c:1\rightarrow[\chi'(G)]$  be a proper edge-colouring of $G.$ Now we show that for every pair ofvertices $u,v\in V(G)$, there is at least one rainbow path. If $uv\in E(G)$, then $uv$ is the rainbow path between the two vertices $u,v$. On the other hand, if $uv\notin E(G)$, there is at least one vertex, say $w$, such that $w\in N_G(u)$ and $w\in N_G(v),$ since $diam(G)=2$. Clearly, $c(uw)\neq c(wv)$, since $G$ is proper. Hence, $uwv$ is the rainbow path connecting two vertices $u,v$. We conclude that $G$ is properly rainbow connected. Thus, $prc(G)\leq\chi'(G)$.

With the aid of Proposition \ref{Prop_Lower_Bau}, we are now able to obtain that $prc(G)=\chi'(G)$.
\end{proof}

By using Proposition \ref{Propo_diam}, we determine proper rainbow connection numbers of some graphs whose diameter equals $2$.

First, we determine the proper rainbow connection number for wheels.

\begin{proposition}
\label{Propo_wheel}
For each integer $n\geq4$, $prc(W_n)=n$
\end{proposition}
\begin{proof}
Suppose that a wheel $W_n$ of order $n+1$ consists of a cycle $C_{n}=v_1v_2\ldots v_nv_1$ and a single vertex $w$ joined to all vertices of cycle $C_n$.

We assign colours $c:E(W_n)\rightarrow[n]$ to all the edges of the wheel $W_n$ as follows: $c(wv_i)=i$, $c(v_iv_{i+1})=i \mod(n)+2$ for $\forall i\in[n-2]$, $c(v_{n-1}v_n)=1$, $c(v_nv_1)=2$. It can be readily seen that $W_n$ is properly coloured. Hence, $\chi'(W_n)\leq n$. On the other hand, $\chi'(W_n)\geq\Delta(W_n)=n$. We deduce that $\chi'(W_n)=n$.

Clearly, $diam(W_n)=2$. By using Proposition \ref{Propo_diam}, $prc(W_n)=\chi'(n)=n$. We obtain the result.
\end{proof}

Next we determine the proper rainbow connection number for the complete bipartite graph $K_{s,t}.$

\begin{theorem}
\label{Thm_Bipartite}
Let $s,t$ be two integers. If $K_{s,t}$ is a complete bipartite graph, then $prc(K_{s,t})=\max\lbrace s,t\rbrace$.
\end{theorem}

Let us mention the following result which is very useful to prove our Theorem \ref{Thm_Bipartite}.

\begin{theorem}[K\"{o}nig et al. \cite{Konig1916}]
\label{Thm_Konig_Chi_Bipar}
If $G$ is bipartite, then $G$ is in Class 1.
\end{theorem}

Now we are able to prove Thereom \ref{Thm_Bipartite}.

\begin{proof}
With the aid of Theorem \ref{Thm_Konig_Chi_Bipar}, it can be readily seen that $\chi'(K_{s,t})=\Delta(K_{s,t})=\max\lbrace s,t\rbrace$. On the other hand, $diam(K_{s,t})=2$. By using Proposition \ref{Propo_diam}, $prc(K_{s,t})=\chi'(K_{s,t})$.

We conclude that $prc(K_{s,t})=\max\lbrace s,t\rbrace$.
\end{proof}

We know that, the chromatic index $\chi'(G)$ depends on the property of $G$ being overfull or not overfull. $G$ is called \textit{overfull} if the number of vertices $n$ is odd, and the number of edges $m$ is greater than $\frac{1}{2}\bigtriangleup(G)(n-1)$. Let us mention Hoffman's result et al. \cite{Hoffman92} on the chromatic index of complete multipartite graph.

\begin{lemma}[Hoffman et al. \cite{Hoffman92}]
\label{Lem_Hoffman_complete_multi}
Let $G$ be a complete multipartite graph. Then $\chi'(G)=\bigtriangleup(G)$ if $G$ is not overfull. Otherwise, $\chi'(G)=\bigtriangleup(G)+1$.
\end{lemma}

Now the, proper connection number of a complete multipartite graph is determined as follows.

\begin{proposition}
\label{Propo_Multipartite}
Let $G$ be a complete multipartite graph. If $G$ is overfull, then $prc(G)=\bigtriangleup(G)+1$. Otherwise, $prc(G)=\bigtriangleup(G)$.
\end{proposition}
\begin{proof}
Suppose that $G$ is a complete multipartite graph. It can be readily seen that $diam(G)=2$. By using Proposition \ref{Propo_diam}, $prc(G)=\chi'(G)$.

Now, applying Lemma \ref{Lem_Hoffman_complete_multi}, $\chi'(G)=\bigtriangleup(G)+1$ if $G$ is overfull or $\chi'(G)=\bigtriangleup(G)$ if $G$ is not overfull. Hence, the result is obtained.
\end{proof}

In \cite{Chartrand2008}, Chartrand et al. showed that $rc(G)=1$ if and only if $G$ is complete. After that, Caro et al. \cite{Caro2008} investigated graphs with small rainbow connection numbers and they gave a sufficient condition that guarantees $rc(G)=2$.

\begin{theorem}[Caro et al. \cite{Caro2008}]
Let $G$ be a nontrivial, connected graph of minimum degree $\delta(G)$. If $\delta(G)\geq\frac{n}{2}+log_2n$, then $rc(G)=2$
\end{theorem}
Next, we show that dense graphs have large proper rainbow connection number.
\begin{proposition}
\label{Propo_minimum}
Let $G$ be a proper edge-coloured graph of order $n$ and minimum degree $\delta(G)$. If $\delta(G)\geq\frac{n-1}{2}$, then $prc(G)= \chi'(G).$
\end{proposition}
\begin{proof}
We show that $diam(G) \leq 2.$ Let $u,v \in V(G)$ be two non adjacent vertices. Then
$d(u) + d(v) = |N(u) \cup N(V| + |N(u) \cap N(v)| \geq 2 \cdot \frac{n-1}{2} = n-1.$ Since $|N(u) \cup N(v)| \leq n-2,$ we conclude $|N(u) \cap N(v)| \geq 1.$ Hence there is a proper coloured path $uwv$ for a vertex $w \in N(u) \cap N(v).$ This shows that
$diam(G) \leq 2.$ Now $prc(G) = \chi'(G)$ follows by Proposition \ref{Propo_diam}.


\end{proof}

This proof immediately leads to the following extension.

\begin{proposition}
\label{Propo_minimumdegreesum1}
Let $G$ be a proper edge-coloured graph of order $n \geq 3.$ If $d(u) + d(v) \geq n-1$ for every pair of non adjacent vertices $u,v \in V(G),$
then $prc(G)= \chi'(G).$
\end{proposition}

\begin{proposition}
\label{Propo2_minimum}
Let $G$ be a proper edge-coloured graph of order $n \geq 9$ and minimum degree $\delta(G)$. If $\delta(G)\geq\frac{n-2}{2}$, then $prc(G)= \chi'(G).$
\end{proposition}
\begin{proof}
Let $u,w \in V(G)$ be any two vertices. If $d(u,w)\leq 2,$ then $u$ and $w$ are connected by a rainbow path of length at most two. Hence we may assume that $d(u,w) \geq 3.$ Then $N[u] \cup N[w] = V(G)$ implying $\delta(G) = \frac{n-2}{2}.$ Since $G$ is connected we conclude that $d(u,w)=3.$ We may assume $4 \leq d(u) \leq d(w),$ since $n \geq 9.$ Let $U = N(u) = \{u_1, u_2, \ldots, u_{d(u)}\}$ and $W = N(w) = \{w_1, w_2, \ldots, w_{d(w)}\}.$
Suppose $u_iw_j, u_iw_k \in E(G).$ Then at least one of the two paths $uu_iw_jw$ and $uu_iw_kw$ is a rainbow $uw$-path. By symmetry we conclude that $E(U,W)$ is an induced matching. Suppose $u_1w_1 \in E(G),$ but $uu_1w_1w$ is no rainbow $uw$-path. We may assume that $c(uu_1)=c(w_1w)=1, c(u_1w_1)=2.$ Since $N(u_1) \cap W = \{w_1\},$ we conclude that
$|N(u_1) \cap (U \setminus \{u_1\})| \geq \delta(G)-2 \geq \frac{n-6}{2} \geq 2.$ We may assume that $u_1u_2,u_1u_3 \in E(G).$ Now $uu_2u_1w_1w$ or $uu_3u_1w_1w$ is a rainbow $uw$-path. Hence $G$ is rainbow connected.
Now $prc(G) = \chi'(G)$ follows by Proposition \ref{Propo_diam}.
\end{proof}

\noindent{\bf Sharpness:} For $n=8$ consider the following graph $F_8$ with vertices $V(F_8)=\{u,u_1,u_2,u_3,w,w_1,w_2,w_3\}$ and edges
$E(F_8)=\{uu_1,uu_2,uu_3,u_1u_2,u_2u_3,u_1w_1,u_3w_3,$ $w_1w_2,w_2w_3,ww_1,ww_2,ww_3\}.$ Then $\chi'(F_8)=3$ and we may assume that
$c(uu_i)=i$ for $1 \leq i \leq 3.$ Then the colours of five further edges are uniquely determined as follows:
$c(u_1u_2)=3, c(u_2u_3)=1, c(u_1w_1)=c(u_3w_3)=c(ww_2)=2.$
For the four remaining edges we obtain (1) $c(w_2w_3)=c(ww_1)=3$ and $c(w_1w_2)=c(ww_3)=1$ or
(2) $c(w_2w_3)=c(ww_1)=1$ and $c(w_1w_2)=c(ww_3)=3.$

This is a proper edge-colouring of $F_8,$ but $F_8$ is not rainbow connected. If we recolour in (1) the edges $u_2u_3$ and $w_2w_3$ by colour $4,$ then $F_8$ becomes proper rainbow connected. Hence $prc(F_8)=4.$ If we switch in (1) the colours of $u_1u_2, u_2u_3$ and $w_1w_2, w_2w_3,$ then $F_8$ becomes
rainbow connected showing that $rc(F_8)=3 = diam(F_8)$.

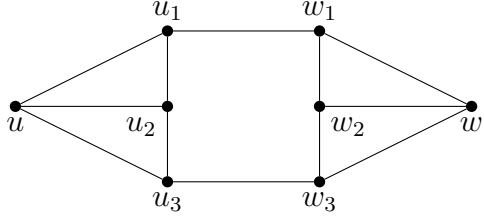
\begin{figure}
\begin{center}
\begin{tikzpicture}
\node [bnode] (u) at (0,0){};
\node [below] at (0,0) {$u$};
\node [bnode] (u1) at (2,1){};
\node [above] at (2,1) {$u_1$};
\node [bnode] (u2) at (2,0){};
\node [below left] at (2,0) {$u_2$};
\node [bnode] (u3) at (2,-1){};
\node [below] at (2,-1){$u_3$};
\node [bnode] (w1) at (4,1){};
\node [above] at (4,1) {$w_1$};
\node [bnode] (w2) at (4,0){};
\node [below right] at (4,0) {$w_2$};
\node [bnode] (w3) at (4,-1){};
\node [below] at (4,-1) {$w_3$};
\node [bnode] (w) at (6,0){};
\node [below] at (6,0) {$w$};
\draw (u)--(u1)--(w1);
\draw (u)--(u2);
\draw (u)--(u3)--(w3);
\draw (u1)--(u2)--(u3);
\draw (w)--(w1);
\draw (w)--(w2);
\draw (w)--(w3);
\draw (w1)--(w2)--(w3);
\end{tikzpicture}
\caption{Graph $F_8$ with $\chi'(F_8)=3$ but $prc(G)=4$}
\end{center}
\end{figure}

This proof immediately leads to the following extension.

\begin{proposition}
\label{Propo_minimumdegreesum2}
Let $G$ be a proper edge-coloured graph of order $n \geq 9.$ If $d(u) + d(v) \geq n-2$ for every pair of non adjacent vertices $u,v \in V(G),$
then $prc(G)= \chi'(G).$
\end{proposition}

\begin{proposition}
\label{Propo3_minimum}
Let $G$ be a proper edge-coloured graph of order $n \geq 9$ and minimum degree $\delta(G)$. If $\delta(G)\geq\frac{n+k}{3}$ for an integer $k \geq 3$, then $prc(G)= \chi'(G).$
\end{proposition}
\begin{proof}
Let $u,w \in V(G)$ be any two vertices. If $d(u,w)\leq 2,$ then $u$ and $w$ are connected by a rainbow path of length at most two. Hence we may assume that $d(u,w) \geq 3.$\\
\noindent{\bf Case 1} $d(u,w)=3$\\
Let $R = V(G) \setminus (N[u] \cup N[w]).$ Then $|R| \leq n - 2(\delta + 1).$ Following arguments in the previous proof we conclude that $E(U,W)$ is an induced matching. Then
$\delta \leq d(u_1) \leq |R|+4 \leq n - 2\delta +2$ implying $\delta \leq \frac{n}{3},$ a contradiction.\\
\noindent{\bf Case 2} $d(u,w)=4$\\
Let $uu_1xw_1w$ be a $uw$-path of length $4,$ and let $R = V(G) \setminus (N[u] \cup N[w]).$ Thus $x \in R$ and $|R| \leq n - 2(\delta +1).$
If $|N(x) \cap (U \cup W)| \geq 5,$ then there is always a rainbow $uw$-path $uu_ixw_jw$ for two vertices $u_i$ and $w_j.$ Hence we may assume that $|N(x) \cap (U \cup W)| \leq 4$ implying
$\delta - 4 \leq d(x)-4 \leq |R|-1 \leq n - 2\delta -3.$  This gives $\delta \leq \frac{n+1}{3},$ a contradiction.\\
\noindent{\bf Case 3} $d(u,w)=5$\\
Let $uu_1x_1x_2w_1w$ be a $uw$-path of length $5,$ and let $R = V(G) \setminus (N[u] \cup N[w]).$ Thus $x_1,x_2 \in R$ and $|R| \leq n - 2(\delta +1).$
Note that $N(x_1) \cap W = N(x_2) \cap U = \emptyset.$
If $|N(x_1) \cap U| \leq 5,$ then $\delta - 5 \leq d(x)-5 \leq |R|-1 \leq n - 2\delta -3.$  This gives $\delta \leq \frac{n+2}{3},$ a contradiction.
Hence we may assume that $|N(x_1) \cap U \geq 6.$ Now there is always a rainbow $uw$-path $uu_ix_1x_2w_1w$ for a vertex $u_i.$  \\
\noindent{\bf Case 4} $d(u,w)=t \geq 6$\\
Let $uu_1x_1x_2 \ldots x_{t-3}w_1w$ be a $uw$-path of length $t.$ Then $N[x_2] \cap N[u] = N[x_2] \cap N[w] = \emptyset$ implying $3(\delta +1) \leq n,$ a contradiction.
\end{proof}

\section{Lower bound}
In this section, we consider the lower bound of proper rainbow connection in dense graphs and some conditions on size of graphs.

\subsection{Dense graphs}
Dense graphs tend to have a small rainbow connection number. However, dense graphs have a large proper rainbow connection number. This follows immediately from its average degree.
\begin{proposition}
\label{Propo_average_degree}
Let $G$ be a nontrivial, connected graph of order $n\geq2$ and average degree $\overline{d}(G)=\frac{2\vert E(G)\vert}{n}$. Then $prc(G)\geq\lceil\overline{d}(G)\rceil$.
\end{proposition}
\begin{proof}
It can be readily seen that $prc(G)\geq\bigtriangleup(G)$ since $prc(G)\geq\chi'(G)$ by Proposition \ref{Prop_Lower_Bau} and $\chi'(G)\geq\bigtriangleup(G)$ by Vizing's Theorem. On the other hand, $\bigtriangleup(G)\geq\lceil\overline{d}(G)\rceil$. Hence, $prc(G)\geq\lceil\overline{d}(G)\rceil$.

We obtain the result.
\end{proof}

\begin{proposition}
\label{Propo1_size_graph}
Let $G$ be a connected graph of order $n$ and size $\vert E(G)\vert\geq k\frac{n}{2}$. Then $prc(G)\geq k.$
\end{proposition}
\begin{proof}
By the handshaking lemma we obtain $\frac{2\vert E(G)\vert}{n}=\overline{d}(G)\geq k$. Now the result follows by Proposition
\ref{Propo_average_degree}.
\end{proof}

\subsection{Size of graphs}
The problem of rainbow connection depending on size of graphs are studied by Kemnitz and Schiermeyer in \cite{Kemnitz2011} as follows: For every integer $k$ with $1\leq k\leq n-1$, compute and minimize the function $f(n,k)$ with the following property: If $\vert E(G)\vert\geq f(n,k)$, then $rc(G)\leq k$, where
$$f(n,k)\geq {{n-k+1}\choose 2}+(k-1).$$
It has been shown in \cite{Kemnitz2011,LLS,KPSW} that equality holds for $k=1,2,3,4,n-6,n-5,n-4,n-3,n-2,n-1.$
Now, we obtain the following result.

\begin{theorem}
\label{theorem_size_graph}
Let $G$ be a connected proper edge-coloured graph of order $n \geq 3.$ If
$$|E(G)| \geq {n-2 \choose 2} + 2,$$
then $prc(G)= \chi'(G)$ or $G \in \{P_4, Z_2, G_{6.3}\}.$
\end{theorem}

\begin{proof}
Let $G$ be a connected graph of order $n \geq 3$ and size $|E(G)| \geq {n-2 \choose 2} + 2.$ Then
$\Delta(G) \geq \lceil\overline{d}(G)\rceil \geq \lceil n-5+\frac{10}{n}\rceil \geq n-4.$


Let $w \in V(G)$ be a vertex with $d(w) = \Delta(G),$ and let $N(w) = W = \{w_1, w_2, \ldots, w_{\Delta(G)}\}$ be its neighours.
We distinguish four cases.

\begin{enumerate}
\item
If $\Delta(G)=n-1,$ then $G[\{ww_i | 1 \leq i \leq n-1\}]$ induces a spanning subgraph $H$ of $G,$ which is rainbow-connected.
Hence $G$ is rainbow-connected.
\item
If $\Delta(G) = n-2,$ then let $V(G) \setminus N[w] = \{u\}.$ First observe that $N(u) \subset N(w).$ So we may assume that
$uw_i \in E(G)$ for $1 \leq i \leq d(u).$ If $d(u) \geq 2,$ then $G[\{ww_i | 1 \leq i \leq n-2\} \cup \{uw_1, uw_2\}]$ induces a spanning subgraph $H$ of $G,$ which is rainbow-connected.
Hence $G$ is rainbow-connected. If $d(u)=1,$ let $c(w_1u)=1, c(ww_1)=2.$ If $c(ww_i) \neq 1$ for $2 \leq i \leq n-2,$ then $G$ is rainbow-connected. Hence we may assume that $c(ww_2)=1.$ Then we are sure that all pairs of vertices $x,y \in V(G)$ are rainbow-connected except for the pair $(u, w_2),$ if $d(u,w_2) \geq 3.$ Hence we may assume that $w_1w_2 \notin E(G).$ Suppose there is a vertex $w_i \in N(w_1) \cap N(w_2)$ for some $3 \leq i \leq n-2.$ We may assume that $c(w_1w_i)=3.$
Then $c(w_2w_i) \neq 1,3$ and so $uw_1w_iw_2$ is a rainbow path. Suppose there is no such vertex $w_i.$ Then
$|E(\overline{G})| \geq 1 + (n-3) + 1 + (n-4) = 2n-5,$ which implies that $|E(G)| = {n-2 \choose 2} +2.$ Thus we have $d(u)=1$ and $d(w_1)+d(w_2)=3+(n-4)=n-1.$ Therefore,
$G-\{w_1,w_2,u\} \cong K_{n-3}.$ We may further assume that $w_2w_i \in E(G)$ for $3 \leq i \leq d(w_2)+1.$



Suppose first that $d(w_2) \geq 3.$ If $c(w_2w_3)=2,$ then $uw_1ww_4w_2$ is a rainbow $uw_2$-path. If
$c(w_2w_3) \neq 2,$ then $uw_1ww_3w_2$ is a rainbow $uw_2$-path.
Suppose next that $d(w_2)=2,$ which implies $n \geq 5.$ If $c(w_2w_3) \neq 2,$ then $uw_1ww_3w_2$ is a rainbow $uw_2$-path. Hence we may assume that
$c(w_2w_3)=2.$ Then $uw_1w_4w_3w_2$ is a rainbow $uw_2$-path for $n \geq 6.$
If $n=5,$ then $G \cong Z_2.$ Note that $\chi'(Z_2)=rc(Z_2=3,$ but $prc(Z_2)=4.$ Hence $Z_2$ is an exceptional graph.

Finally suppose that $d(w_2)=1.$ Then $G$ consists of a complete graph of order $n-2$ induced by $V(G) \setminus \{u,w_2\}$ and two pendant edges attached at $w$ and $w_1.$
If $n \geq 4$ is odd, then $\chi'(G)=n-2=\Delta(G).$ Observe that the $K_{n-2}$
is coloured with $n-2$ colours and that $w$ and $w_1$ have distinct palettes of colours for their incident edges. Hence the pendant edges $ww_2$ and $w_1u$ have distinct colours. This shows that $G$ is rainbow connected.
If $n \geq 4$ is even, then $G \cong P_4$ for $n = 4$ implying $prc(G)=rc(G)=3.$ Therefore, $P_4$ is an exceptional graph, since
$\chi'(P_4)=2.$

If $n \geq 6,$
take an optimal edge colouring of the $K_{n-2}$ using $n-3$ colours. Now switch the colour from $ww_3$ to the edge $ww_2,$
and colour the edges $ww_3$ and $w_1u$ with a new colour. Now observe that this colouring makes $G$ properly rainbow connected.\\

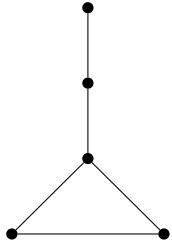
\begin{figure}
\begin{center}
\begin{tikzpicture}
\node [bnode] (u1) at (0,0){};
\node [bnode] (u2) at (-1,-1){};
\node [bnode] (u3) at (1,-1){};
\node [bnode] (u4) at (0,1){};
\node [bnode] (u5) at (0,2){};
\draw (u1)--(u2)--(u3)--(u1);
\draw (u1)--(u4)--(u5);
\end{tikzpicture}
\caption{Graph $Z_2$}
\end{center}
\end{figure}

\item
If $\Delta(G) = n-3,$ then let $V(G) \setminus N[w] = U = \{u_1,u_2\}.$ We first distinguish two cases.

\noindent{\bf Case 1} $u_1u_2 \in E(G)$\\
We first show that $G - w_i$ is connected for $1 \leq i \leq n-3.$ Suppose that
$G - w_i$ is disconnected for some $i$ with $1 \leq i \leq n-3.$ Then
$|E(G)| \leq |E(G[N[w]-w_i])|+|E(G[\{u_1,u_2\}])|+d(w_i) \leq {n-3 \choose 2} + 1 + (n-3) < {n-2 \choose 2} + 2,$
a contradiction.
Now $G-w_1$ and $G-w_2$ both have size at least ${n-2 \choose 2} + 2 - (n-3) = {(n-1)-2 \choose 2} + 2.$ So by induction both
$G-w_1$ and $G-w_2$ are rainbow connected or an exceptional graph. Furthermore $w_1$ and $w_2$ are rainbow connected, since $1 \leq d(w_1,w_2) \leq 2.$ This shows that $G$ is rainbow connected or an exceptional graph.

\noindent{\bf Case 2} $u_1u_2 \notin E(G)$\\
Let $d(u_1) \geq d(u_2) \geq 1.$ Then $d(w) + \sum_{i=1}^{n-3}d(w_i) \leq (n-3) + (n-3)(n-3) - d(u_1)-d(u_2)$ implying
$|E(G)| \leq {n-2 \choose 2} + \frac{d(u_1)+d(u_2)}{2} < {n-2 \choose 2} + 2$ for $d(u_1)+d(u_2) \leq 3,$ a contradiction.
Hence we may assume that $d(u_1) + d(u_2) \geq 4.$ Now if $d(u_1) \geq 3, d(u_2)=1$ or $d(u_1) \geq 2, d(u_2) \geq 2,$ then there are always
two vertices $w_i, w_j$ such that $G - w_i$ and $G - w_j$ are both connected. This shows that $G$ is rainbow connected or an exceptional graph.

This discussion shows, that in both cases $G$ is either properly rainbow connected or an exceptional graph.
So suppose that $m(G - w_i) = {n-3 \choose 2} + 2$ and $G - w_i$ is not properly rainbow connected for some $1 \leq i \leq d(w).$ We may choose the labeling of the vertices of
$G$ such that $G - w_1$ is not properly rainbow connected.

If $n=5,$ then $G - w_1 \cong P_4.$ Taking into account that $d(w_1)=2=\Delta(G)$ we conclude that $G \cong C_5.$ Now we have
$prc(C_5)=3=rc(C_5)=\chi'(C_5).$

If $n=6,$ then $G - w_1 \cong Z_2.$ Now up to isomorphism the following three graphs $G_{6.1},G_{6.2}$ and $G_{6.3}$ are possible.
\begin{enumerate}
\item
$G_{6.1}: c(ww_2)=c(w_1w_3)=c(u_1u_2)=1, c(ww_1)=c(w_2w_3)=2, c(ww_3)=c(w_1u_1)=c(w_2u_2)=3.$\\
This shows that $prc(G_{6.1})=rc(G_{6.1})=\chi'(G_{6.1})=3.$
\item
$G_{6.2}: c(ww_2)=c(w_1w_3)=c(u_1u_2)=1, c(ww_1)=c(w_2w_3)=2, c(ww_3)=c(w_1u_1)=3, c(w_2u_1)=4.$\\
Observe that $\chi'(G_{6.2})=4.$ This shows that $prc(G_{6.2})=\chi'(G_{6.2})=4,$ whereas $rc(G_{6.2})=3.$
\item
For $G_{6.3}$ we can show that $rc(G_{6.3})=\chi'(G_{6.3})=3,$ whereas $prc(G_{6.3})=4.$ Up to a permutation of the colours $G_{6.3}$ has to be coloured as follows:
$G_{6.3}: c(ww_1)=c(w_2w_3)=c(u_1u_2)=1, c(ww_2)=c(w_1u_2)=2, c(ww_3)=c(w_2u_2)=c(w_1u_1)=3.$ Thus $G_{6.3}$ is an exceptional graph.
\end{enumerate}

If $n=7,$ then $G - w_1 \cong G_{6.3}.$ Observe that $d(w_1)=4.$ Now we can always find two vertices $x,y \in V(G_{6.3})$ such that $d(x,y) \leq 2,$ $G - x$ and
$G - y$ are connected, and $\Delta(G-x)=\Delta(G-y)=4.$ Hence, $G$ is no exceptional graph. No by induction it follows that there are no exceptional graphs
$G$ with size ${n-2 \choose 2} + 2$ for all $n \geq 7.$

\begin{figure}
\begin{minipage}{0.28\textwidth}
\begin{center}
\begin{tikzpicture}
\node [bnode] (w3) at (0,0){};
\node [below] at (0,0) {$w_3$};
\node [bnode] (w1) at (-1.5,0){};
\node [below left] at (-1.5,0) {$w_1$};
\node [bnode] (w2) at (1.5,0){};
\node [below right] at (1.5,0) {$w_2$};
\node [bnode] (w) at (0,1.5){};
\node [above] at (0,1.5) {$w$};
\node [bnode] (u1) at (-0.75,-1.5) {};
\node [below] at (-0.75,-1.5) {$u_1$};
\node [bnode] (u2) at (0.75,-1.5) {};
\node [below] at (0.75,-1.5) {$u_2$};
\draw (w1)--(w)--(w2)--(u2)--(u1)--(w1)--(w3)--(w2);
\draw (w3)--(w);
\node [above] at (0,-1.5) {\scriptsize $1$};
\node [below left] at (-1,-0.75) {\scriptsize $3$};
\node [below right] at (1,-0.75) {\scriptsize $3$};
\node [below] at (-0.75,0) {\scriptsize $1$};
\node [below] at (0.75,0) {\scriptsize $2$};
\node [above left] at (-0.75,0.5) {\scriptsize $2$};
\node [above right] at (0.75,0.5) {\scriptsize $1$};
\node [left] at (0,0.75) {\scriptsize $3$};
\end{tikzpicture}
\caption{Graph $G_{6.1}$}
\end{center}
\end{minipage}
\begin{minipage}{0.4\textwidth}
\begin{center}
\begin{tikzpicture}
\node [bnode] (w3) at (0,0){};
\node [below] at (0,0) {$w_3$};
\node [bnode] (w1) at (-1.5,0){};
\node [below left] at (-1.5,0) {$w_1$};
\node [bnode] (w2) at (1.5,0){};
\node [below right] at (1.5,0) {$w_2$};
\node [bnode] (w) at (0,1.5){};
\node [above] at (0,1.5) {$w$};
\node [bnode] (u1) at (-0.75,-1.5) {};
\node [below] at (-0.75,-1.5) {$u_1$};
\node [bnode] (u2) at (0.75,-1.5) {};
\node [below] at (0.75,-1.5) {$u_2$};
\draw (w1)--(w)--(w2)--(u1)--(w1)--(w3)--(w2);
\draw (w3)--(w);
\draw (u1)--(u2);
\node [below] at (0,-1.5) {\scriptsize $1$};
\node [below left] at (-1.1,-0.55) {\scriptsize $3$};
\node [above] at (-0.75,0) {\scriptsize $1$};
\node [above] at (0.75,0) {\scriptsize $2$};
\node [above left] at (-0.75,0.5) {\scriptsize $2$};
\node [above right] at (0.75,0.5) {\scriptsize $1$};
\node [left] at (0,0.75) {\scriptsize $3$};
\node [below left] at (0.85,-0.55) {\scriptsize $4$};
\end{tikzpicture}
\caption{Graph $G_{6.2}$}
\end{center}
\end{minipage}
\begin{minipage}{0.28\textwidth}
\begin{center}
\begin{tikzpicture}
\node [bnode] (w2) at (0,0){};
\node [below left] at (0,0) {$w_2$};
\node [bnode] (w1) at (-1.5,0){};
\node [above left] at (-1.5,0) {$w_1$};
\node [bnode] (w3) at (1.5,0){};
\node [above right] at (1.5,0) {$w_3$};
\node [bnode] (w) at (0,1.5){};
\node [above] at (0,1.5) {$w$};
\node [bnode] (u1) at (-0.75,-1.5) {};
\node [below] at (-0.75,-1.5) {$u_1$};
\node [bnode] (u2) at (0.75,-1.5) {};
\node [below] at (0.75,-1.5) {$u_2$};
\draw (w1)--(w)--(w3)--(w2)--(u2)--(u1)--(w1);
\draw (w1)--(u2);
\draw (w)--(w2);
\node [below] at (0,-1.5) {\scriptsize $1$};
\node [below left] at (-1.1,-0.55) {\scriptsize $3$};
\node [above] at (0.75,0) {\scriptsize $1$};
\node [above left] at (-0.75,0.5) {\scriptsize $1$};
\node [above right] at (0.75,0.5) {\scriptsize $3$};
\node [left] at (0,0.75) {\scriptsize $2$};
\node [below left] at (0.85,-0.55) {\scriptsize $3$};
\node [below] at (0,-0.55) {\scriptsize $2$};
\end{tikzpicture}
\caption{Graph $G_{6.3}$}
\end{center}
\end{minipage}
\end{figure}

\item
If $\Delta(G) = n-4,$ then $n \geq 7.$ Let $V(G) \setminus N[w] = U = \{u_1,u_2,u_3\}.$ We distinguish three cases.

\noindent{\bf Case 1} $U$ is connected\\
We first show that $G - w_i$ is connected for $1 \leq i \leq n-4.$ Suppose that
$G - w_i$ is disconnected for some $i$ with $1 \leq i \leq n-4.$ Then
$|E(G)| \leq |E(G[N[w]-w_i])|+|E(G[\{u_1,u_2,u_3\}])|+d(w_i) \leq {n-4 \choose 2} + 3 + (n-4)
= {n-3 \choose 2} + 3 < {n-2 \choose 2} + 2,$
a contradiction. Hence we may assume that $|N(U) \cap W| \geq 2.$

\medskip
\noindent{\bf Fact 1:} Then there are always two vertices $w_i, w_j$ such that $G - w_i$ and $G - w_j$ are both connected.
Now $G-w_i$ and $G-w_j$ both have size at least ${n-2 \choose 2} + 3 - (n-4) = {(n-1)-2 \choose 2} + 4.$ So by induction both
$G-w_i$ and $G-w_j$ are rainbow connected. Furthermore $w_1$ and $w_2$ are rainbow connected, since $1 \leq d(w_1,w_2) \leq 2.$ This shows that $G$ is rainbow connected.

\noindent{\bf Case 2} $|E(U)|=1$\\
We may assume that $E(U)=u_1u_2.$ If $|E(U,W)| \leq 3,$ then $\sum_{v \in V(G)}d(v) \leq (n-3)(n-4)+2+3$ implying
$m(G) \leq {n-3 \choose 2} + 2 < {n-2 \choose 2} + 2,$ a contradiction. Hence we may assume that $|E(U,W)| \geq 4.$ Now
considering the two components $\{u_1, u_2\}$ and $u_3$ of $U$ as two vertices we can follow the previous Case 2 for $\Delta(G)=n-3.$

\noindent{\bf Case 3} $E(U)= \emptyset$\\
Let $d(u_1) \geq d(u_2) \geq d(u_3) \geq 1.$ If $d(u_3) \geq 2,$ then there are always two vertices $w_i, w_j$ such that $G - w_i$ and
$G - w_j$ are connected. Moreover, $d(w_i, w_j) \leq 2$ and we apply Fact 1.
So we may assume that $d(u_3)=1.$ Let $u_3w_1 \in E(G).$ If $|E(\{u_1, u_2\},W-w_1| \geq 4,$ then there are two vertices $w_i, w_j, 2 \leq i < j \leq j$ such that
$G- w_i$ and $G - w_j$ are connected and we apply Fact 1.
Hence we may assume that $|E(U,W)| \leq 3 + 3 = 6.$ This implies $|E(G)| \leq {n-3 \choose 2} + 6 < {n-2 \choose 2} + 2,$ a contradiction.


\end{enumerate}
\end{proof}

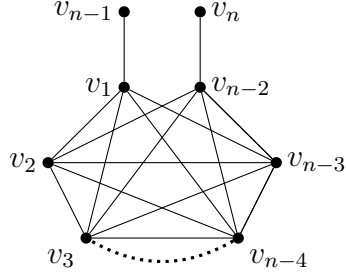
\begin{figure}
\begin{center}
\begin{tikzpicture}
\node [bnode] (v1) at (0,0) {};
\node [left] at (0,0) {$v_1$};
\node [bnode] (v2) at (-1,-1) {};
\node [left] at (-1,-1) {$v_2$};
\node [bnode] (v3) at (-0.5,-2) {};
\node [below left] at (-0.5,-2) {$v_3$};
\node [bnode] (v4) at (1.5,-2) {};
\node [below right] at (1.5,-2) {$v_{n-4}$};
\node [bnode] (v5) at (2,-1) {};
\node [right] at (2,-1) {$v_{n-3}$};
\node [bnode] (v6) at (1,0) {};
\node [right] at (1,0) {$v_{n-2}$};
\draw (v1)--(v2)--(v3);
\draw (v4)--(v5)--(v6);
\draw (v1)--(v3);
\draw (v1)--(v4);
\draw (v1)--(v5);
\draw (v2)--(v4);
\draw (v2)--(v5);
\draw (v2)--(v6);
\draw (v3)--(v4);
\draw (v3)--(v5);
\draw (v3)--(v6);
\draw (v4)--(v5);
\draw (v4)--(v6);
\draw (v5)--(v6);
\draw (v4) [dotted, very thick, bend left] to (v3);
\node [bnode] (v7) at (0,1) {};
\node [left] at (0,1) {$v_{n-1}$};
\node [bnode] (v8) at (1,1) {};
\node [right] at (1,1) {$v_n$};
\draw (v1)--(v7);
\draw (v6)--(v8);
\end{tikzpicture}
\caption{Graph $F_n$, where $n$ is even}
\end{center}
\end{figure}

\noindent{\bf Sharpness:} For even $n \geq 6$ take a complete graph of order $n-2$ and label its vertices $v_1, v_2, \ldots, v_{n-2}.$ Now we add two vertices $v_{n-1}, v_n,$ add the edges $v_1v_{n-1},v_{n-2}v_n,$ and delete the edge $v_1v_{n-2}.$ Let $F_n$ denote the resulting graph. Then $|E(F_n)|= {n-2 \choose 2}+1, \Delta(F_n)=\chi'(F_n)=n-3, rc(F_n)=diam(F_n)=4,$ but $prc(G)=n-2,$ if $n$ is even.
This can be seen as follows. In any edge colouring of $F_n$ using $n-3$ colours, $F_n - \{v_{n-2}, v_{n-1}, v_n\}$ uses all $n-3$ colours, each on $\frac{n-4}{2}$ edges. Applying the same argument on $F_n - \{v_1, v_{n-1}, v_n\},$ we deduce that $v_1$ and $v_{n-2}$ have the same palette of $n-4$ colours for their incident edges in $F_n - \{v_{n-1},v_n\}.$ Suppose these colours are
$1, 2, \ldots, n-4.$ Then both edges $v_1v_{n-1}$ and $v_{n-2}v_n$ obtain colour $n-3.$ But then there is no rainbow
$v_{n-1}v_n$-path.



\end{document}